\documentclass[a4paper,12pt]{article}
\usepackage{amsmath,amsthm,amssymb}
\usepackage{amsmath}
\usepackage{amsfonts}
\usepackage[all]{xy}
\usepackage{amscd}

\advance \textheight by \topskip
\advance \textheight by 1pt
\makeatother

\numberwithin{equation}{section}

\newtheorem{thm}{Theorem}[section]
\newtheorem{lem}[thm]{Lemma}
\newtheorem{cor}[thm]{Corollary}
\newtheorem{prop}[thm]{Proposition}

 {\theoremstyle{definition}

\newtheorem{rem}[thm]{Remark}
}

\title{Reduction of bielliptic hyperelliptic functions of genus 3}
\author{Takanori Ayano\footnote{Osaka Central Advanced Mathematical Institute, Osaka Metropolitan University, \newline \hspace{3ex} 3-3-138, Sugimoto, Sumiyoshi-ku, Osaka, 558-8585, Japan. \newline \hspace{3ex} Email: ayano@omu.ac.jp
\newline \hspace{3ex} 2020 Mathematics Subject Classification: Primary 14H42, Secondary 14K25, 14H40 
\newline \hspace{3ex} Key words and phrases: reduction of hyperelliptic functions, elliptic function, hyperelliptic sigma function, Jacobian variety of an algebraic curve
}}

\date{}

\begin{document}
\maketitle

\begin{abstract}
The present paper is devoted to the problem about the reduction of hyperelliptic functions of genus 3. 
Our research was motivated by applications to the theory of equations and dynamical systems integrable in hyperelliptic functions. 
In this paper, we consider a hyperelliptic curve of genus 3 which admits a morphism of degree 2 to an elliptic curve. 
We express the hyperelliptic functions associated with the curve of genus 3 in terms of the Weierstrass elliptic functions and hyperelliptic functions of genus 2. 
\end{abstract}

\section{Introduction}

The elliptic sigma function, which is defined and studied by Weierstrass, is generalized to the sigma functions associated with hyperelliptic curves and many applications in integrable systems and mathematical physics are derived
(cf. \cite{BEL-97-1, BEL-97-2, BEL-2012}). 
For example, it is well known that the hyperelliptic functions defined by the logarithmic derivatives of the hyperelliptic sigma functions satisfy the KdV hierarchy and the KP equation.  

The problem whether the Jacobian variety of a hyperelliptic curve is isogenous to the direct product of the Jacobian varieties of hyperelliptic curves of lower genera is considered in many papers 
(e.g., \cite{B-W-2003, BCMS2021, B-L-S-2013, FK, Howe-Leprevost-Poonen-Large-torsion-subgroups-2000, KT, Ku, P, J.P.Serre-2020}).
This problem is naturally connected with the following well-known problem: Suppose that solutions of differential equations and dynamical systems are given in terms of hyperelliptic functions. 
Under conditions when a reduction of these functions to hyperelliptic functions of lower genera is possible, find an explicit form of these solutions in terms of the hyperelliptic functions of lower genera. 

The reduction of the Riemann period matrices is studied by Weierstrass and Poincar\'e. 
This is summarized in \cite{BE0, E1}. 
The reduction of the Riemann theta functions is studied in \cite{BE0, E1, AOSmirnov1988}. 
The reduction of hyperelliptic integrals to elliptic integrals is studied in many papers (e.g., \cite{BE0, E1, BE2, O-Bolza-1887, O-Bolza-1935}).
In particular, O. Bolza derived many examples of the reduction of hyperelliptic integrals to elliptic integrals (cf.~\cite{O-Bolza-1887, O-Bolza-1935}). 
This problem is closely related to that of a morphism from a hyperelliptic curve to an elliptic curve. 
In this paper, we consider the reduction of hyperelliptic functions. 
A curve is called \textit{bielliptic} if it admits a morphism of degree 2 to an elliptic curve. 
In \cite{Enolskii-Salerno-1996}, the hyperelliptic functions associated with a bielliptic hyperelliptic curve of genus 2 are expressed in terms of the Jacobi elliptic functions. 
In \cite{AB2022}, the hyperelliptic functions associated with a bielliptic hyperelliptic curve of genus 2 are expressed in terms of the Weierstrass elliptic functions. 
The purpose of this paper is to express the hyperelliptic functions associated with a bielliptic hyperelliptic curve of genus 3 in terms of the Weierstrass elliptic functions and hyperelliptic functions of genus 2.

For a hyperelliptic curve $\mathcal{X}$ of genus $g$, let $\sigma_{\mathcal{X}}(u)$ with $u={}^t(u_1,u_3,\dots,u_{2g-1})\in\mathbb{C}^g$ be the sigma function associated with $\mathcal{X}$, which is a holomorphic function on $\mathbb{C}^g$
 (cf. \cite{BEL-97-1, BEL-97-2, BEL-2012}). 
Let $\wp_{i,j}^{\mathcal{X}}=-\partial_{u_j}\partial_{u_i}\log\sigma_{\mathcal{X}}$ and $\wp_{i,j,k}^{\mathcal{X}}=\partial_{u_k}\wp_{i,j}^{\mathcal{X}}$, where $\partial_{u_l}=\partial/\partial u_l$. 
Let $\operatorname{Jac}(\mathcal{X})$ be the Jacobian variety of $\mathcal{X}$. 
The functions $\wp_{i,j}^{\mathcal{X}}$ and $\wp_{i,j,k}^{\mathcal{X}}$ are meromorphic functions on $\operatorname{Jac}(\mathcal{X})$.  
Any meromorphic function on $\operatorname{Jac}(\mathcal{X})$ has a rational representation in terms of $\wp_{1,i}^{\mathcal{X}}$ and $\wp_{1,1,i}^{\mathcal{X}}$ with $i=1,3,\dots,2g-1$. 
It is well known that the function $\wp_{1,1}^{\mathcal{X}}$ gives a solution of the KdV hierarchy (cf. \cite{BEL-97-1, BEL-97-2, BEL-2012}). 
For $\alpha, \beta, \gamma\in\mathbb{C}$ satisfying $\beta\gamma\neq0$, we consider the polynomial 
\[f(x)=x(x-1)(x-\alpha^2)(x-\beta^2)(x-\gamma^2)(x-\alpha^2/\beta^2)(x-\alpha^2/\gamma^2).\]
We assume that $f(x)$ has no multiple roots. 
We consider the non-singular hyperelliptic curve of genus $3$
\[V=\Bigl\{(x,y)\in\mathbb{C}^2 \Bigm| y^2=f(x)\Bigr\}.\]
A hyperelliptic curve of genus 3 is bielliptic if and only if the curve is isomorphic to a curve in the form of $V$ (see Section \ref{2024.12.24.1}). 
Let 
\[k_1=\frac{\sqrt{-1}\beta}{\sqrt{(1-\beta^2)(\beta^2-\alpha^2)}},\qquad k_2=\frac{4\sqrt{-1}\alpha\beta\gamma}{(1-\alpha)(\beta^2-\alpha)(\gamma^2-\alpha)}.\]
We consider the elliptic curve $E$ and the hyperelliptic curve $C$ of genus 2 defined by 
\begin{align*}
E&=\biggl\{(X,Y)\in\mathbb{C}^2 \biggm| Y^2=X(X-1)\Bigl(X-k_1^2(1-1/\gamma^2)(\gamma^2-\alpha^2)\Bigr)\biggr\},\\
C&=\Bigl\{(X,Y)\in\mathbb{C}^2 \Bigm| Y^2=X(X-1)(X-a^2)(X-b^2)(X-c^2)\Bigr\},
\end{align*}
where 
\[a=\frac{1+\alpha}{1-\alpha},\qquad b=\frac{\beta^2+\alpha}{\beta^2-\alpha},\qquad c=\frac{\gamma^2+\alpha}{\gamma^2-\alpha}.\]
Then we can define the morphisms of degree 2 (see Section \ref{2024.12.24.1})
\begin{align*}
&\varphi_1\colon\quad V\to E,\qquad (x,y)\mapsto(X,Y)=\Biggl(\frac{k_1^2(x-1)(x-\alpha^2)}{x}, \frac{k_1^3y}{x^2}\Biggr),\\
&\varphi_2\colon\quad V\to C,\qquad (x,y)\mapsto(X,Y)=\Biggl(\biggl(\frac{x+\alpha}{x-\alpha}\biggr)^2, \frac{4\alpha k_2(x+\alpha)y}{(x-\alpha)^5}\Biggr).
\end{align*}
The maps $\varphi_1$ and $\varphi_2$ induce the homomorphisms of the Jacobian varieties
\begin{gather*}
\varphi_1^*\colon\ \operatorname{Jac}(E)\to\operatorname{Jac}(V),\qquad
\varphi_2^*\colon\ \operatorname{Jac}(C)\to\operatorname{Jac}(V). 
\end{gather*}
In this paper, we describe the maps $\varphi_1^*$ and $\varphi_2^*$ explicitly in Lemma \ref{2.15.1}.
For simplicity, we denote $\wp_{1,1}^E$ by $\wp$. Let $\wp'=\frac{d}{du}\wp$. 
For $u\in \operatorname{Jac}(E)$, the functions $\wp_{i,j}^V\bigl(\varphi_1^*(u)\bigr)$ and $\wp_{i,j,k}^V\bigl(\varphi_1^*(u)\bigr)$ are meromorphic functions on $\operatorname{Jac}(E)$. 
In this paper, for $i=1,3,5$, we describe the functions $\wp_{1,i}^V\bigl(\varphi_1^*(u)\bigr)$ and $\wp_{1,1,i}^V\bigl(\varphi_1^*(u)\bigr)$ in terms of $\wp(u)$ and $\wp'(u)$ explicitly in Theorem \ref{2023.2.21.2}, i.e.,
we show that the restrictions of the hyperelliptic functions $\wp_{1,i}^V$ and $\wp_{1,1,i}^V$ to the appropriate subspace of dimension $1$ in $\mathbb{C}^3$ are elliptic functions and describe them in terms of the Weierstrass elliptic functions associated with $E$. 
For $u\in \operatorname{Jac}(C)$, the functions $\wp_{i,j}^V\bigl(\varphi_2^*(u)\bigr)$ and $\wp_{i,j,k}^V\bigl(\varphi_2^*(u)\bigr)$ are meromorphic functions on $\operatorname{Jac}(C)$. 
In \cite{B-K-2020}, the values of the hyperelliptic functions of genus $g$ on the image of $g+1$ points with respect to the Abel-Jacobi map are expressed in terms of rational functions of the coordinates of these points. 
By using this formula, for $i=1,3,5$, we describe the functions $\wp_{1,i}^V\bigl(\varphi_2^*(u)\bigr)$ and $\wp_{1,1,i}^V\bigl(\varphi_2^*(u)\bigr)$ in terms of $\wp_{1,j}^C(u)$ and $\wp_{1,1,j}^C(u)$ with $j=1,3$ explicitly in Theorem \ref{2023.2.10.1}, i.e.,
we show that the restrictions of the hyperelliptic functions $\wp_{1,i}^V$ and $\wp_{1,1,i}^V$ to the appropriate subspace of dimension $2$ in $\mathbb{C}^3$ are hyperelliptic functions of genus $2$ and describe them in terms of the hyperelliptic functions associated with $C$.  
In \cite{BL-2005}, the addition formula for the hyperelliptic functions of genus $g$ are given. 
By using this addition formula, for $i=1,3,5$, we show that it is possible to express $\wp_{1,i}^V$ and $\wp_{1,1,i}^V$ on $\mathbb{C}^3$ in terms of the elliptic functions associated with $E$ and the hyperelliptic functions of genus $2$ associated with $C$ explicitly in Corollary \ref{2023.2.21.3}. 

The reductions of the Riemann theta functions and the hyperelliptic functions are important in real physical problems (e.g., \cite{AP-2006, BE0, E1, BE2, FCoppini2020, Enolskii-Salerno-1996, Grinevich2024, AOSmirnov1988, AOSmirnov}). 
In \cite{Matsutani2024}, a graphical representation of the hyperelliptic functions of genus 2 is given. 
In \cite{Bernatska2024}, numerical computations for the hyperelliptic functions are studied.  
By using the reduction formulae for hyperelliptic functions, it is easier to compute the values of the hyperelliptic functions of higher genera. 
I believe that our results of this paper will contribute to numerical computations and graphical representations of the hyperelliptic functions. 

The present paper is organized as follows. 
In Section 2, we review the definition and properties of the hyperelliptic sigma functions. 
In Section 3, we give a normal form of a bielliptic hyperelliptic curve of genus 3 and maps from the bielliptic hyperelliptic curve of genus 3 to an elliptic curve and a hyperelliptic curve of genus 2 explicitly. 
In Section 4, we describe the maps between the Jacobian varieties induced by the maps from the bielliptic hyperelliptic curve of genus 3 to the elliptic curve and the hyperelliptic curve of genus 2 explicitly. 
In Section 5, we express the hyperelliptic functions associated with the bielliptic hyperelliptic curve of genus 3 in terms of the elliptic functions and the hyperelliptic functions of genus 2.

\section{The hyperelliptic sigma functions}\label{2022.8.29.1}

In this section, we review the definition of the hyperelliptic sigma functions and give facts about them which will be used later on. 
For details of the hyperelliptic sigma functions, see \cite{BEL-97-1, BEL-97-2, BEL-2012}.  

We set 
\[N(x)=x^{2g+1}+\lambda_2x^{2g}+\lambda_4x^{2g-1}+\cdots+\lambda_{4g}x+\lambda_{4g+2}, \qquad \lambda_i\in\mathbb{C}.\]
We assume that $N(x)$ has no multiple roots. 
We consider the non-singular hyperelliptic curve of genus $g$
\[\mathcal{X}=\Bigl\{(x,y)\in\mathbb{C}^2 \Bigm| y^2=N(x)\Bigr\}.\]  
A basis of the vector space consisting of holomorphic 1-forms on $\mathcal{X}$ is given by 
\[
\omega_{2i-1}^{\mathcal{X}}=-\frac{x^{g-i}}{2y}dx, \qquad 1\le i\le g. 
\]
We set $\omega_\mathcal{X}={}^t(\omega_1^{\mathcal{X}},\omega_3^{\mathcal{X}},\dots,\omega_{2g-1}^{\mathcal{X}})$. 
We consider the following meromorphic 1-forms on $\mathcal{X}$:  
\begin{equation}
\eta_{2i-1}^{\mathcal{X}}=-\frac{1}{2y}\sum_{k=g-i+1}^{g+i-1}(k+i-g)\lambda_{2g+2i-2k-2}x^kdx, \qquad 1\le i\le g,\label{dr}
\end{equation}
which are holomorphic at any point except $\infty$. 
In (\ref{dr}), we set $\lambda_0=1$. For example, for $g=1$ we have 
\[\eta_1^{\mathcal{X}}=-\frac{x}{2y}dx,\]
for $g=2$ we have 
\[\eta_1^{\mathcal{X}}=-\frac{x^2}{2y}dx, \qquad \eta_3^{\mathcal{X}}=-\frac{\lambda_4x+2\lambda_2x^2+3x^3}{2y}dx,\]
and for $g=3$ we have 
\begin{gather*}
\eta_1^{\mathcal{X}}=-\frac{x^3}{2y}dx,\qquad \eta_3^{\mathcal{X}}=-\frac{\lambda_4x^2+2\lambda_2x^3+3x^4}{2y}dx,\\
\eta_5^{\mathcal{X}}=-\frac{\lambda_8x+2\lambda_6x^2+3\lambda_4x^3+4\lambda_2x^4+5x^5}{2y}dx.
\end{gather*}
Let $\{\mathfrak{a}_i,\mathfrak{b}_i\}_{i=1}^g$ be a canonical basis in the one-dimensional homology group of the curve $\mathcal{X}$. 
We define the period matrices by 
\begin{gather*}
2\omega'=\Biggl(\int_{\mathfrak{a}_j}\omega_{2i-1}^{\mathcal{X}}\Biggr),\qquad 2\omega''=\Biggl(\int_{\mathfrak{b}_j}\omega_{2i-1}^{\mathcal{X}}\Biggr),\\
-2\eta'=\Biggl(\int_{\mathfrak{a}_j}\eta_{2i-1}^{\mathcal{X}}\Biggr), \qquad -2\eta''=\Biggl(\int_{\mathfrak{b}_j}\eta_{2i-1}^{\mathcal{X}}\Biggr).
\end{gather*}
We define the period lattice $\Lambda=\bigl\{2\omega'm_1+2\omega''m_2\bigm | m_1,m_2\in\mathbb{Z}^g\bigr\}$ and consider the Jacobian variety $\operatorname{Jac}(\mathcal{X})=\mathbb{C}^g/\Lambda$. 
Let $\pi_{\mathcal{X}}\colon \mathbb{C}^g\to\mbox{Jac}(\mathcal{X})$ be the natural projection.  
The normalized period matrix is given by $\tau=(\omega')^{-1}\omega''$. 
Let $\tau\delta'+\delta''$ with $\delta',\delta''\in\mathbb{R}^g$ be the Riemann constant with respect to $\bigl(\{\mathfrak{a}_i,\mathfrak{b}_i\}_{i=1}^g,\infty\bigr)$. We denote the imaginary unit by $\textbf{i}$.
The sigma function $\sigma_{\mathcal{X}}(u)$ associated with the curve $\mathcal{X}$, $u={}^t(u_1, u_3,\dots, u_{2g-1})\in\mathbb{C}^g$, is defined by
\[
\sigma_{\mathcal{X}}(u)=\varepsilon\exp\biggl(\frac{1}{2}{}^tu\eta'(\omega')^{-1}u\biggr)\theta\begin{bmatrix}\delta'\\ \delta'' \end{bmatrix}\bigl((2\omega')^{-1}u,\tau\bigr),
\]
where $\theta\begin{bmatrix}\delta'\\ \delta'' \end{bmatrix}(u,\tau)$ is the Riemann theta function with the characteristics $\begin{bmatrix}\delta'\\ \delta'' \end{bmatrix}$ defined by
\[
\theta\begin{bmatrix}\delta'\\ \delta'' \end{bmatrix}(u,\tau)=\sum_{n\in\mathbb{Z}^g}\exp\bigl\{\pi{\bf i}\,{}^t(n+\delta')\tau(n+\delta')+2\pi{\bf i}\,{}^t(n+\delta')(u+\delta'')\bigr\}
\]
and $\varepsilon$ is a non-zero constant. 

\begin{prop}[{\cite[pp.~7--8]{BEL-97-1}}]
For $m_1,m_2\in\mathbb{Z}^g$, let $\Omega=2\omega'm_1+2\omega''m_2$.
Then, for $u\in\mathbb{C}^g$, we have
\begin{align*}
&\sigma_{\mathcal{X}}(u+\Omega)/\sigma_{\mathcal{X}}(u) \\ 
&=(-1)^{2({}^t\delta'm_1-{}^t\delta''m_2)+{}^tm_1m_2}\exp\bigl\{{}^t(2\eta'm_1+2\eta''m_2)(u+\omega'm_1+\omega''m_2)\bigr\}.
\end{align*}
\end{prop}

Let $\wp_{i,j}^{\mathcal{X}}=-\partial_{u_j}\partial_{u_i}\log\sigma_{\mathcal{X}}$ and $\wp_{i,j,k}^{\mathcal{X}}=\partial_{u_k}\wp_{i,j}^{\mathcal{X}}$, where $\partial_{u_l}=\partial/\partial u_l$.

\begin{rem}
The functions $\wp_{i,j}^{\mathcal{X}}$ and $\wp_{i,j,k}^{\mathcal{X}}$ are meromorphic functions on $\operatorname{Jac}(\mathcal{X})$. 
Any meromorphic function on $\operatorname{Jac}(\mathcal{X})$ has a rational representation in terms of $\wp_{1,i}^{\mathcal{X}}$ and $\wp_{1,1,i}^{\mathcal{X}}$ with $i=1,3,\dots,2g-1$. 
\end{rem}

\begin{rem}
In \cite[Theorem 4.12]{BEL-97-1}, it is proved that the function $2\wp_{1,1}^{\mathcal{X}}+2\lambda_2/3$ satisfies the KdV hierarchy. 
\end{rem}

\section{Bielliptic hyperelliptic curves of genus 3}\label{2024.12.24.1}

A curve is called \textit{bielliptic} if it admits a morphism of degree 2 to an elliptic curve. 
For $a,b,c\in\mathbb{C}$, we consider the polynomial 
\[F(s)=(s^2-1)(s^2-a^2)(s^2-b^2)(s^2-c^2).\] 
We assume that $F(s)$ has no multiple roots. 
We consider the non-singular hyperelliptic curve of genus $3$
\[H=\Bigl\{(s,t)\in\mathbb{C}^2 \Bigm| t^2=F(s)\Bigr\}.\]
It is known that a hyperelliptic curve of genus 3 is bielliptic if and only if the curve is isomorphic to a curve in the form of $H$ (cf. \cite[Lemma 2.2]{S1} and \cite[Section 3]{KT}). 
We consider the elliptic curve $W$ and the hyperelliptic curve $C$ of genus 2 defined by 
\begin{align*}
W&=\Bigl\{(S,T)\in\mathbb{C}^2 \Bigm | T^2=(S-1)(S-a^2)(S-b^2)(S-c^2)\Bigr\},\\
C&=\Bigl\{(X,Y)\in\mathbb{C}^2 \Bigm | Y^2=X(X-1)(X-a^2)(X-b^2)(X-c^2)\Bigr\}.
\end{align*}
Then we have the morphisms of degree 2
\begin{align*}
&\phi_1\colon\quad H\to W,\qquad (s,t)\mapsto(S,T)=(s^2,t),\\
&\phi_2\colon\quad H\to C,\qquad (s,t)\mapsto(X,Y)=(s^2, st)
\end{align*}
(cf. \cite[Lemma 2.2]{S1} and \cite[Section 3]{KT}). 
It is known that the Jacobian variety of $H$ is isogenous to the direct product of the Jacobian varieties of $W$ and $C$ (cf. \cite[Theorem 5]{P} and \cite[Section 3]{KT}).    
For $\alpha, \beta, \gamma\in\mathbb{C}$ satisfying $\beta\gamma\neq0$, we consider the polynomial 
\begin{equation}
f(x)=x(x-1)(x-\alpha^2)(x-\beta^2)(x-\gamma^2)(x-\alpha^2/\beta^2)(x-\alpha^2/\gamma^2).\label{2023.2.21.1}
\end{equation}
We assume that $f(x)$ has no multiple roots. 
We consider the non-singular hyperelliptic curve of genus $3$
\[V=\Bigl\{(x,y)\in\mathbb{C}^2 \Bigm | y^2=f(x)\Bigr\}.\]
Let 
\[k_1=\frac{\sqrt{-1}\beta}{\sqrt{(1-\beta^2)(\beta^2-\alpha^2)}},\qquad k_2=\frac{4\sqrt{-1}\alpha\beta\gamma}{(1-\alpha)(\beta^2-\alpha)(\gamma^2-\alpha)}.\]

\begin{prop}
Given $a, b, c\in\mathbb{C}$ such that $F(s)$ has no multiple roots, the curve $H$ is isomorphic to the curve $V$ with 
\[\alpha=\frac{a-1}{a+1},\qquad \beta=\sqrt{\frac{(a-1)(b+1)}{(a+1)(b-1)}},\qquad \gamma=\sqrt{\frac{(a-1)(c+1)}{(a+1)(c-1)}}\]
by the morphism 
\[\zeta\colon\quad H\to V,\qquad (s,t)\mapsto (x,y)\]
with 
\[
(x,y)=\Biggl(\frac{(a-1)(s+1)}{(a+1)(s-1)},\;\frac{8(a-1)^3t}{(a+1)^4\sqrt{(1-b^2)(c^2-1)}(s-1)^4}\Biggr).
\]
Conversely, given $\alpha, \beta, \gamma\in\mathbb{C}$ such that $\beta\gamma\neq0$ and $f(x)$ has no multiple roots, the curve $V$ is isomorphic to the curve $H$ with 
\begin{equation}
a=\frac{1+\alpha}{1-\alpha},\qquad b=\frac{\beta^2+\alpha}{\beta^2-\alpha},\qquad c=\frac{\gamma^2+\alpha}{\gamma^2-\alpha}\label{4.26}
\end{equation}
by the morphism 
\[\widetilde{\zeta}\colon\quad V\to H,\qquad (x,y)\mapsto(s,t)=\Biggl(\frac{x+\alpha}{x-\alpha},\;\frac{4\alpha k_2y}{(x-\alpha)^4}\Biggr).\]
\end{prop}

\begin{proof}
By the direct calculations, we obtain the statement of the proposition. 
\end{proof}

The elliptic curve $W$ with (\ref{4.26}) is isomorphic to the elliptic curve in Legendre form
\[E=\biggl\{(X,Y)\in\mathbb{C}^2 \biggm | Y^2=X(X-1)\Bigl(X-k_1^2(1-1/\gamma^2)(\gamma^2-\alpha^2)\Bigr)\biggr\}\]
by the morphism 
\[\xi\colon\quad W\to E,\;\;\;\;(S,T)\mapsto(X,Y)\]
with 
\[(X,Y)=\Biggl(k_1^2\frac{(\alpha+1)^2-(\alpha-1)^2S}{S-1},\;\frac{4\alpha k_1^3T}{k_2(S-1)^2}\Biggr).\]
We consider the hyperelliptic curve $C$ of genus 2 with (\ref{4.26}) and the maps
\[\varphi_1=\xi\circ\phi_1\circ\widetilde{\zeta}\colon\quad V\to E,\qquad \varphi_2=\phi_2\circ\widetilde{\zeta}\colon\quad V\to C.\]
Then the maps $\varphi_i$ with $i=1,2$ are described by 
\begin{align}
&\varphi_1\colon\quad V\to E,\qquad (x,y)\mapsto(X,Y)=\Biggl(\frac{k_1^2(x-1)(x-\alpha^2)}{x}, \frac{k_1^3y}{x^2}\Biggr),\label{2023.1.9.1}\\
&\varphi_2\colon\quad V\to C,\qquad (x,y)\mapsto(X,Y)=\Biggl(\biggl(\frac{x+\alpha}{x-\alpha}\biggr)^2, \frac{4\alpha k_2(x+\alpha)y}{(x-\alpha)^5}\Biggr).\label{2023.2.8.1}
\end{align}

\section{The maps between the Jacobian varieties}

Let $M_{m,n}(\mathbb{C})$ be the set of the $m\times n$ matrices such that all the components are complex numbers. 
We have $\varphi_1(\infty)=\infty$. 
We consider the maps 
\begin{align*}
&A\colon\quad V\to\operatorname{Jac}(V),\qquad P\mapsto \int_{\infty}^P\omega_V,\\
&A_1\colon\quad E\to\operatorname{Jac}(E),\qquad P\mapsto\int_{\infty}^P\omega_E.
\end{align*}
It is known that we have the unique homomorphism
\[\varphi_{1,*}\colon\quad \operatorname{Jac}(V)\to \operatorname{Jac}(E)\]
such that $A_1\circ\varphi_1=\varphi_{1,*}\circ A$ (cf. \cite[p. 104, Proposition 6.1]{Milne} ). 
\begin{equation}
\begin{split}
\xymatrix{
\operatorname{Jac} (V)  \ar[r]^{\varphi_{1,*}} &\operatorname{Jac}(E) \\
V \ar[u]^{A} \ar[r]^{\varphi_1} &E \ar[u]^{A_1}. 
}\label{2022.9.2.119}
\end{split}
\end{equation}
It is known that we have the unique linear map 
\[\widetilde{\varphi}_{1,*}\colon\quad\mathbb{C}^3\to\mathbb{C},\qquad u\mapsto K_1u,\qquad K_1\in M_{1,3}(\mathbb{C})\]
such that $\varphi_{1,*}\circ\pi_V=\pi_E\circ\widetilde{\varphi}_{1,*}$ (e.g., \cite[Proposition~1.2.1]{B-H-2004}). 
\begin{gather*}
 \begin{CD}
 \mathbb{C}^3 @>{\widetilde{\varphi}_{1,*}}>> \mathbb{C} \\
 @V{\pi_V}VV @V{\pi_E}VV \\
 \operatorname{Jac}(V) @>{\varphi_{1,*}} >> \operatorname{Jac}(E). 
 \end{CD}
\end{gather*}

\begin{lem}\label{3.6.1}
We have $K_1=\displaystyle{\frac{1}{k_1}(1, 0, -\alpha^2)}$. 
\end{lem}

\begin{proof}
Let $K_1=(a_1, a_3, a_5)$. 
We consider the commutative diagram (\ref{2022.9.2.119}). 
The pullback of the holomorphic 1-form $du_1$ on $\operatorname{Jac}(E)$ with respect to $\varphi_{1,*}$ is $a_1du_1+a_3du_3+a_5du_5$.
The pullback of $a_1du_1+a_3du_3+a_5du_5$ with respect to $A$ is $a_1\omega_1^V+a_3\omega_3^V+a_5\omega_5^V$.
On the other hand, the pullback of $du_1$ with respect to $A_1$ is $\omega_E$. 
The pullback of $\omega_E$ with respect to $\varphi_1$ is $\frac{1}{k_1}\omega_1^V-\frac{\alpha^2}{k_1}\omega_5^V$. 
Since $\omega_1^V, \omega_3^V, \omega_5^V$ are linearly independent, we obtain the statement of the lemma. 
\end{proof}

We consider the map
\[A_2\colon\quad C\to\operatorname{Jac}(C),\qquad P\mapsto\int_{(1,0)}^P\omega_C.\]
It is known that we have the unique homomorphism 
\[\varphi_{2,*}\colon\quad \operatorname{Jac}(V)\to \operatorname{Jac}(C)\]
such that $A_2\circ\varphi_2=\varphi_{2,*}\circ A$ (cf. \cite[p. 104, Proposition 6.1]{Milne}). 
\begin{equation}
\begin{split}
\xymatrix{
\operatorname{Jac} (V)  \ar[r]^{\varphi_{2,*}} &\operatorname{Jac}(C) \\
V \ar[u]^{A} \ar[r]^{\varphi_2} &C \ar[u]^{A_2}. 
}\label{2022.9.2.11}
\end{split}
\end{equation}
It is known that we have the unique linear map 
\[\widetilde{\varphi}_{2,*}\colon\quad\mathbb{C}^3\to\mathbb{C}^2,\qquad u\mapsto K_2u,\qquad K_2\in M_{2,3}(\mathbb{C})\]
such that $\varphi_{2,*}\circ\pi_V=\pi_C\circ\widetilde{\varphi}_{2,*}$ (e.g., \cite[Proposition~1.2.1]{B-H-2004}). 
\begin{gather*}
 \begin{CD}
 \mathbb{C}^3 @>{\widetilde{\varphi}_{2,*}}>> \mathbb{C}^2 \\
 @V{\pi_V}VV @V{\pi_C}VV \\
 \operatorname{Jac}(V) @>{\varphi_{2,*}} >> \operatorname{Jac}(C).
 \end{CD}
\end{gather*}

\begin{lem}\label{3.6.11999}
We have $\displaystyle{K_2=-\frac{1}{k_2}\begin{pmatrix}1&2\alpha&\alpha^2\\1&-2\alpha&\alpha^2\end{pmatrix}}$. 
\end{lem}

\begin{proof}
Let $K_2=\begin{pmatrix}b_1 &b_3 &b_5\\c_1&c_3&c_5\end{pmatrix}$. 
We consider the commutative diagram (\ref{2022.9.2.11}). 
The pullbacks of the holomorphic 1-forms $du_1$ and $du_3$ on $\operatorname{Jac}(C)$ with respect to $\varphi_{2,*}$ are $b_1du_1+b_3du_3+b_5du_5$ and $c_1du_1+c_3du_3+c_5du_5$, respectively. 
The pullbacks of $b_1du_1+b_3du_3+b_5du_5$ and $c_1du_1+c_3du_3+c_5du_5$ with respect to $A$ are $b_1\omega_1^V+b_3\omega_3^V+b_5\omega_5^V$ and $c_1\omega_1^V+c_3\omega_3^V+c_5\omega_5^V$, respectively. 
On the other hand, the pullbacks of $du_1$ and $du_3$ with respect to $A_2$ are $\omega_1^C$ and $\omega_3^C$, respectively. 
The pullbacks of $\omega_1^C$ and $\omega_3^C$ with respect to $\varphi_2$ are $-\frac{1}{k_2}\omega_1^V-\frac{2\alpha}{k_2}\omega_3^V-\frac{\alpha^2}{k_2}\omega_5^V$ and 
$-\frac{1}{k_2}\omega_1^V+\frac{2\alpha}{k_2}\omega_3^V-\frac{\alpha^2}{k_2}\omega_5^V$, respectively. 
Since $\omega_1^V, \omega_3^V, \omega_5^V$ are linearly independent, we obtain the statement of the lemma. 
\end{proof}

We have $\varphi_1^{-1}(\infty)=\{\infty, (0,0)\}$ and $\varphi_2^{-1}\bigl((1,0)\bigr)=\{\infty, (0,0)\}$.  
It is well known that the maps $\varphi_i$ with $i=1,2$ induce the homomorphisms between the Jacobian varieties 
\[\varphi_1^*\colon\quad \operatorname{Jac}(E)\to\operatorname{Jac}(V),\qquad
\int_{\infty}^P\omega_E\mapsto \int_{\infty}^{P_1}\omega_V+\int_{(0,0)}^{P_2}\omega_V,\]
where $P\in E$ and $\varphi_1^{-1}(P)=\{P_1,P_2\}$, 
\[\varphi_2^*\colon\quad \operatorname{Jac}(C)\to\operatorname{Jac}(V),\qquad
\int_{(1,0)}^Q\omega_C\mapsto \int_{\infty}^{Q_1}\omega_V+\int_{(0,0)}^{Q_2}\omega_V,\]
where $Q\in C$ and $\varphi_2^{-1}(Q)=\{Q_1,Q_2\}$.  
It is known that we have the unique linear maps 
\begin{align*}
&\widetilde{\varphi}_1^*\colon\quad\mathbb{C}\to\mathbb{C}^3,\qquad u\mapsto L_1u,\qquad L_1\in M_{3,1}(\mathbb{C}),\\
&\widetilde{\varphi}_2^*\colon\quad\mathbb{C}^2\to\mathbb{C}^3,\qquad u\mapsto L_2u,\qquad L_2\in M_{3,2}(\mathbb{C})
\end{align*}
such that $\varphi_1^*\circ\pi_E=\pi_V\circ\widetilde{\varphi}_1^*$ and $\varphi_2^*\circ\pi_C=\pi_V\circ\widetilde{\varphi}_2^*$ (e.g., \cite[Proposition~1.2.1]{B-H-2004}). 
\begin{gather*}
 \begin{CD}
 \mathbb{C} @>{\widetilde{\varphi}_1^*}>> \mathbb{C}^3 \\
 @V{\pi_E}VV @V{\pi_V}VV \\
 \operatorname{Jac}(E) @>{\varphi_1^*} >> \operatorname{Jac}(V), 
 \end{CD}\qquad
\begin{CD}
 \mathbb{C}^2 @>{\widetilde{\varphi}_2^*}>> \mathbb{C}^3 \\
 @V{\pi_C}VV @V{\pi_V}VV \\
 \operatorname{Jac}(C) @>{\varphi_2^*} >> \operatorname{Jac}(V). 
 \end{CD}
\end{gather*}
Let 
\begin{align*}
&\varphi_*\colon\quad \operatorname{Jac}(V)\to \operatorname{Jac}(E) \times \operatorname{Jac}(C),\qquad u\mapsto \bigl(\varphi_{1,*}(u), \varphi_{2,*}(u)\bigr),\\
&\varphi^*\colon\quad \operatorname{Jac}(E) \times \operatorname{Jac}(C) \to\operatorname{Jac}(V),\qquad (u,v)\mapsto \varphi_1^*(u)+\varphi_2^*(v). 
\end{align*}

\begin{lem}[{\cite[Section 3]{KT}}]\label{3.9.1}
For $u\in\operatorname{Jac}(V)$, we have
\begin{gather*}
\varphi^*\circ\varphi_*(u)=2u.
\end{gather*}
\end{lem}

\begin{proof}
For the sake to be complete and self-contained, we give a proof of this lemma. 
For $P=(x,y)\in V\backslash\{\infty\}$, let $\varphi_1^{-1}\bigl(\varphi_1(P)\bigr)=\{P, Q\}$ and $\varphi_2^{-1}\bigl(\varphi_2(P)\bigr)=\{P,R\}$. 
If $P\neq(0,0)$, then we have $Q=\Bigl(\frac{\alpha^2}{x}, \frac{\alpha^4y}{x^4}\Bigr)$ and $R=\Bigl(\frac{\alpha^2}{x}, -\frac{\alpha^4y}{x^4}\Bigr)$. 
If $P=(0,0)$, then we have $Q=R=\infty$. 
We have 
\begin{align*}
&\varphi^*\circ\varphi_*\Biggl(\int_{\infty}^P\omega_V\Biggr)=\varphi^*\Biggl(\int_{\infty}^{\varphi_1(P)}\omega_E, \int_{(1,0)}^{\varphi_2(P)}\omega_C\Biggr)\\
&=\int_{\infty}^P\omega_V+\int_{(0,0)}^Q\omega_V+\int_{\infty}^P\omega_V+\int_{(0,0)}^R\omega_V \\
&=2\int_{\infty}^P\omega_V+\int_{\infty}^Q\omega_V+\int_{\infty}^R\omega_V=2\int_{\infty}^P\omega_V.
\end{align*}
For $u\in\operatorname{Jac}(V)$, there exist $Z_1, Z_2, Z_3\in V$ such that 
\[\sum_{i=1}^3\int_{\infty}^{Z_i}\omega_V=u.\]
Since the map $\varphi^*\circ\varphi_*$ is a homomorphism, we obtain the statement of the lemma. 
\end{proof}

\begin{lem}\label{2.15.1}
We have $L_1=k_1\begin{pmatrix}1\\0\\-\alpha^{-2}\end{pmatrix}$ and $L_2=\displaystyle{-\frac{k_2}{2}\begin{pmatrix}1&1\\ \alpha^{-1}&-\alpha^{-1}\\ \alpha^{-2} &\alpha^{-2}\end{pmatrix}}$. 
\end{lem}

\begin{proof}
Let $K$ and $L$ be the $3\times 3$ matrices defined by $K=\begin{pmatrix}K_1\\K_2\end{pmatrix}$ and $L=\begin{pmatrix}L_1&L_2\end{pmatrix}$. 
We consider the maps
\begin{align*}
&\widetilde{\varphi}_*\colon\quad \mathbb{C}^3\to\mathbb{C}^3,\qquad u\mapsto Ku,\\
&\widetilde{\varphi}^*\colon\quad \mathbb{C}^3\to\mathbb{C}^3,\qquad u\mapsto Lu,\\
&\pi_{E,C}\colon\quad \mathbb{C}^3\to\operatorname{Jac}(E)\times\operatorname{Jac}(C),\qquad {}^t(u_1,u_3,u_5)\mapsto\Bigl(\pi_E(u_1), \pi_C\bigl({}^t(u_3,u_5)\bigr)\Bigr).
\end{align*}
Then we have the following commutative diagrams:
\begin{gather*}
 \begin{CD}
 \mathbb{C}^3 @>{\widetilde{\varphi}_*}>> \mathbb{C}^3 \\
 @V{\pi_V}VV @V{\pi_{E,C}}VV \\
 \operatorname{Jac}(V) @>{\varphi_*} >>\operatorname{Jac}(E)\times\operatorname{Jac}(C),
 \end{CD}\qquad
\begin{CD}
 \mathbb{C}^3 @>{\widetilde{\varphi}^*}>> \mathbb{C}^3 \\
 @V{\pi_{E,C}}VV @V{\pi_V}VV \\
 \operatorname{Jac}(E)\times\operatorname{Jac}(C) @>{\varphi^*} >> \operatorname{Jac}(V).
 \end{CD}
\end{gather*}
Therefore, we obtain the following commutative diagram:
\begin{gather*}
 \begin{CD}
 \mathbb{C}^3 @>{\widetilde{\varphi}^*\circ\widetilde{\varphi}_*}>> \mathbb{C}^3 \\
 @V{\pi_{V}}VV @V{\pi_{V}}VV \\
 \operatorname{Jac}(V) @>{\varphi^*\circ\varphi_*} >> \operatorname{Jac}(V).
 \end{CD}
\end{gather*}
We have
\[
\widetilde{\varphi}^*\circ\widetilde{\varphi}_*\colon\quad \mathbb{C}^3\to\mathbb{C}^3,\qquad u\mapsto LKu.
\]
From Lemma \ref{3.9.1}, we obtain
\[
LK=\begin{pmatrix}2&0&0\\0&2&0\\0&0&2\end{pmatrix}.
\]
From Lemmas \ref{3.6.1} and \ref{3.6.11999}, we obtain the statement of the lemma.
\end{proof}

\section{Reduction of the hyperelliptic functions associated with $V$}

For simplicity, we denote $\wp_{1,1}^E$ by $\wp$. 
Let $\wp'=\frac{d}{du}\wp$.

\begin{thm}\label{2023.2.21.2}

For $u\in\mathbb{C}$, the following relations hold$:$
\begin{align*}
\wp_{1,1}^V(L_1u)&=\alpha^2+1+\wp(u)/k_1^2,\qquad \wp_{1,3}^V(L_1u)=-\alpha^2,\qquad \wp_{1,5}^V(L_1u)=0,\\
\wp_{1,1,1}^V(L_1u)&=\wp'(u)/k_1^3,\qquad \wp_{1,1,3}^V(L_1u)=0,\qquad \wp_{1,1,5}^V(L_1u)=0.
\end{align*}
\end{thm}

\begin{proof}
Let $U$ be the subset of $E$ consisting of the elements $P\in E\backslash\{\infty\}$ such that $X\neq-k_1^2(\alpha\pm 1)^2$, 
where $P=(X,Y)$. 
Then $U$ is an open set in $E$. 
Let $U'=A_1(U)$. 
Then $U'$ is an open set in $\operatorname{Jac}(E)$. 
We~take a point $u\in U'$. 
Then there exists $P=(X,Y)\in U$ such that $u=A_1(P)$. 
We have
\begin{gather*}
\varphi_1^*(u)=\int_{\infty}^{P_1}\omega_V+\int_{(0,0)}^{P_2}\omega_V=\int_{\infty}^{P_1}\omega_V+\int_{\infty}^{P_2}\omega_V+\int_{\infty}^{(0,0)}\omega_V,
\end{gather*}
where $\varphi_1^{-1}(P)=\{P_1,P_2\}$. 
From $P\neq\infty$, we have $P_1, P_2\neq\infty, (0,0)$.
Let $P_i=(x_i,y_i)$ for $i=1,2$.
From $P\in U$, we have $x_1\neq x_2$. 
Therefore, from the well-known solution of the Jacobi inversion problem, we have
\begin{gather*}
\wp_{1,1}^V(L_1u)=\wp_{1,1}^V\bigl(\varphi_1^*(u)\bigr)=x_1+x_2,\qquad \wp_{1,3}^V(L_1u)=\wp_{1,3}^V\bigl(\varphi_1^*(u)\bigr)=-x_1x_2,\\
\wp_{1,5}^V(L_1u)=\wp_{1,5}^V\bigl(\varphi_1^*(u)\bigr)=0,\\
\begin{pmatrix}2y_1\\2y_2\\0\end{pmatrix}=-\begin{pmatrix}x_1^2&x_1&1\\x_2^2&x_2&1\\0&0&1\end{pmatrix}\begin{pmatrix}\wp_{1,1,1}^V(L_1u)\\\wp_{1,1,3}^V(L_1u)\\\wp_{1,1,5}^V(L_1u)\end{pmatrix}.
\end{gather*}
From (\ref{2023.1.9.1}), $X=\wp(u)$, and $Y=-\wp'(u)/2$, the relations in Theorem \ref{2023.2.21.2} hold for $u\in\pi_E^{-1}(U')$.
Since $\pi_E^{-1}(U')$ is an open set in $\mathbb{C}$, the relations in Theorem \ref{2023.2.21.2} hold on $\mathbb{C}$.
\end{proof}

For $0\le i\le 6$, we define $\mu_{2i}$ by 
\[f(x)=\mu_0x^7+\mu_2x^6+\mu_4x^5+\mu_6x^4+\mu_8x^3+\mu_{10}x^2+\mu_{12}x,\]
where $f(x)$ is defined by (\ref{2023.2.21.1}). 

\begin{thm}\label{2023.2.10.1}
Let $u\in\mathbb{C}^2$. 
For simplicity, we set $p_2=\wp_{1,1}^C(u)$, $p_3=\wp_{1,1,1}^C(u)$, $p_4=\wp_{1,3}^C(u)$, and $p_5=\wp_{1,1,3}^C(u)$.  
The following relations hold$:$ 
\begin{align*}
\wp_{1,1}^V(L_2u)&=\frac{D_1^2}{16k_2^2p_4^2(1-p_2-p_4)^2}-\mu_2-h_1,\\
\wp_{1,3}^V(L_2u)&=\frac{\alpha D_1D_2}{8k_2^2p_4^2(1-p_2-p_4)^3}-\mu_4-\mu_2h_1-h_2,\\
\wp_{1,5}^V(L_2u)&=\frac{\alpha^2(D_3p_3^2p_4^2-2D_4p_3p_4p_5+D_5p_5^2)}{16k_2^2p_4^2(1-p_2-p_4)^4}-\mu_6-\mu_4h_1-\mu_2h_2-h_3,\\
\wp_{1,1,1}^V(L_2u)&=\frac{\alpha(p_2-3p_4+3)(p_2p_5-p_5-p_3p_4)}{2k_2p_4(p_2+p_4-1)}+\wp_{1,1}^V(L_2u)D_6,\\
\wp_{1,1,3}^V(L_2u)&=\frac{\alpha^2(3p_2p_4p_5-p_4p_5+p_2^2p_5+2p_2p_5-3p_5-3p_3p_4^2-p_2p_3p_4-p_3p_4)}{2k_2p_4(p_2+p_4-1)}\\
&\quad +\wp_{1,3}^V(L_2u)D_6,\\
\wp_{1,1,5}^V(L_2u)&=\frac{\alpha^3(p_3p_4-p_2p_5+p_5)}{2k_2p_4}+\wp_{1,5}^V(L_2u)D_6,
\end{align*}
where 
\begin{align*}
D_1&=p_2p_4p_5-3p_4p_5-p_2^2p_5+2p_2p_5-p_5-p_3p_4^2+p_2p_3p_4-3p_3p_4,\\
D_2&=p_2p_4^2p_5+3p_4^2p_5+4p_2^2p_4p_5-14p_2p_4p_5+14p_4p_5-p_2^3p_5+p_2^2p_5\\
&\quad +p_2p_5-p_5-p_3p_4^3-4p_2p_3p_4^2+14p_3p_4^2+p_2^2p_3p_4+3p_3p_4,\\
D_3&=p_4^4+4(p_2+7)p_4^3+2(19p_2^2-118p_2+243)p_4^2\\
&\quad -4(7p_2^3-37p_2^2+77p_2-119)p_4+p_2^4-4p_2^3+6p_2^2+28p_2+33,\\
D_4&=(p_2-1)p_4^4+4(p_2^2+14p_2-47)p_4^3+2(19p_2^3-137p_2^2+329p_2-323)p_4^2\\
&\quad -4(7p_2^4-36p_2^3+74p_2^2-76p_2+47)p_4+(p_2-1)^5,\\
D_5&=(p_2^2-2p_2+33)p_4^4+4(p_2^3+21p_2^2-101p_2+119)p_4^3\\
&\quad +2(19p_2^4-156p_2^3+418p_2^2-492p_2+243)p_4^2-28(p_2-1)^5p_4+(p_2-1)^6,\\
D_6&=\frac{p_2p_4p_5-3p_4p_5-p_2^2p_5+2p_2p_5-p_5-p_3p_4^2+p_2p_3p_4-3p_3p_4}{2k_2p_4(p_2+p_4-1)},\\
h_1&=\frac{4\alpha(p_4+1)}{p_2+p_4-1},\qquad h_2=\frac{2\alpha^2(5p_4^2-2p_2p_4+22p_4+p_2^2+2p_2+5)}{(p_2+p_4-1)^2},\\
h_3&=\frac{4\alpha^3(p_4+1)(5p_4^2-6p_2p_4+54p_4+5p_2^2+6p_2+5)}{(p_2+p_4-1)^3}.
\end{align*}

\end{thm}

\begin{proof}
Let $\mbox{Sym}^2(C)$ be the 2-fold symmetric product of $C$. 
Then $\mbox{Sym}^2(C)$ is a complex manifold of dimension 2. 
We consider the holomorphic map
\begin{gather*}
J\colon\quad \mbox{Sym}^2(C)\to\operatorname{Jac}(C),\qquad (P_1,P_2)\mapsto \sum_{i=1}^2\int_{(1,0)}^{P_i}\omega_C. 
\end{gather*}
Let $U$ be the subset of $\operatorname{Sym}^2(C)$ consisting of the elements $(P_1,P_2)$ such that $P_1,P_2\in C\backslash\{\infty\}$, $X_1,X_2\neq1$, $X_1\neq X_2$, and $X_1X_2\neq0$, where $P_i=(X_i,Y_i)$ for $i=1,2$. 
Then $U$ is an open set in $\operatorname{Sym}^2(C)$. 
Let $U'=J(U)$. 
Then $U'$ is an open set in $\operatorname{Jac}(C)$. 
We~take a point $u\in U'$. 
Then there exists $(P_1,P_2)\in U$ such that $u=J\bigl((P_1,P_2)\bigr)$. 
We have
\begin{align*}
\varphi_2^*(u)&=\int_{\infty}^{S_1}\omega_V+\int_{(0,0)}^{S_2}\omega_V+\int_{\infty}^{S_3}\omega_V+\int_{(0,0)}^{S_4}\omega_V \\ 
&=\int_{\infty}^{S_1}\omega_V+\int_{\infty}^{S_2}\omega_V+\int_{\infty}^{S_3}\omega_V+\int_{\infty}^{S_4}\omega_V,
\end{align*}
where $\varphi_2^{-1}(P_1)=\{S_1,S_2\}$ and $\varphi_2^{-1}(P_2)=\{S_3, S_4\}$. 
From $X_1, X_2\neq1$, we have $S_i\neq\infty$ for $1\le i\le 4$. 
Let $S_i=(x_i,y_i)$ for $1\le i\le 4$. 
From $(P_1,P_2)\in U$, we have $x_i\neq x_j$ for any $i,j$ satisfying $i\neq j$. 
From \cite[Theorem 3.1]{B-K-2020}, we have 
\begin{align*}
\wp_{1,i}^V(L_2u)&=\wp_{1,i}^V\bigl(\varphi_2^*(u)\bigr)=\mathcal{H}_{i+1},\qquad i=1,3,5,\\
\wp_{1,1,i}^V(L_2u)&=\wp_{1,1,i}^V\bigl(\varphi_2^*(u)\bigr)=2\mathcal{L}_{i+2},\qquad i=1,3,5,
\end{align*}
where $\mathcal{H}(x)=x^3-\mathcal{H}_2x^2-\mathcal{H}_4x-\mathcal{H}_6$ with 
\[\mathcal{H}(x)=-\sum_{1\le k<l\le4}\frac{(y_k-y_l)^2}{(x_k-x_l)^2}\prod_{i=1,\;i\neq k,l}^4\frac{x-x_i}{(x_k-x_i)(x_l-x_i)}+\sum_{i=0}^3x^i\sum_{j=0}^{3-i}\mu_{6-2i-2j}\widetilde{h}_j\]
and $\mathcal{L}(x)=\mathcal{L}_3x^2+\mathcal{L}_5x+\mathcal{L}_7$ with 
\[\mathcal{L}(x)=\sum_{i=1}^4y_i\frac{\prod_{j=1,\;j\neq i}^4(x-x_j)-\mathcal{H}(x)}{\prod_{j=1,\;j\neq i}^4(x_i-x_j)}.\]
Here, $\widetilde{h}_j$ denotes the complete symmetric polynomial of degree $j$ in $\{x_i\}_{i=1}^4$. 
From (\ref{2023.2.8.1}) and 
\[X_1+X_2=\wp_{1,1}^C(u),\quad X_1X_2=-\wp_{1,3}^C(u),\quad 2Y_i=-\wp_{1,1,1}^C(u)X_i-\wp_{1,1,3}^C(u),\quad i=1,2,\]
the relations in Theorem \ref{2023.2.10.1} hold for $u\in\pi_C^{-1}(U')$.
Since $\pi_C^{-1}(U')$ is an open set in $\mathbb{C}^2$, the relations in Theorem \ref{2023.2.10.1} hold on $\mathbb{C}^2$.
\end{proof}

For $u\in\mathbb{C}^3$, let $\mathfrak{c}_i(u)$ with $1\le i\le 4$ be the 3-dimensional vectors defined by 
\[\mathfrak{c}_1(u)=\begin{pmatrix}\wp_{1,5}^V(u)\\\wp_{1,3}^V(u)\\\wp_{1,1}^V(u)\end{pmatrix},\quad \mathfrak{c}_2(u)=\begin{pmatrix}\wp_{1,1,5}^V(u)/2\\\wp_{1,1,3}^V(u)/2\\\wp_{1,1,1}^V(u)/2\end{pmatrix},\quad \mathfrak{c}_{i+2}(u)=G(u)\mathfrak{c}_i(u),\quad 
1\le i\le 2,\]
where $G(u)=\begin{pmatrix}0&0&\wp_{1,5}^V(u)\\1&0&\wp_{1,3}^V(u)\\0&1&\wp_{1,1}^V(u)\end{pmatrix}$. 
Let us define the $3\times3$ matrix $B(u)=\bigl(\mathfrak{c}_1(u), \mathfrak{c}_2(u), \mathfrak{c}_3(u)\bigr)$. 
For a square matrix $A$, we denote the determinant of $A$ by $|A|$. 
For $u, v\in\mathbb{C}^3$, let 
\[R(x,y)=\left|\begin{array}{ccc}\begin{matrix}1&x&x^2\end{matrix}&\begin{matrix}x^3&y&x^4\end{matrix}&xy\\I_3&B(u)&\mathfrak{c}_4(u)\\I_3&B(v)&\mathfrak{c}_4(v)\end{array}\right|/|B(v)-B(u)|,\]
where $I_3$ is the identity matrix of size $3$.  
There exists a polynomial $\Phi(w_1,w_2)$ in variables $w_1$ and $w_2$ such that $\Phi(x,y^2)=-R(x,y)R(x,-y)$.  
Let us represent $R(x,y)$ as a linear combination of monomials: 
\[R(x,y)=(x+\mathfrak{d}_2)y+x^3(\mathfrak{d}_1x+\mathfrak{d}_3)+\mathfrak{d}_5x^2+\mathfrak{d}_7x+\mathfrak{d}_9.\]
Let $\mathfrak{d}={}^t(\mathfrak{d}_9, \mathfrak{d}_7, \mathfrak{d}_5)$. 

\begin{thm}[{\cite[Corollary 3.2]{BL-2005}}]\label{2024.9.15.1}
$(\mathrm{i})$ For $u, v\in\mathbb{C}^3$, as a polynomial in $x$, we have the following relation
\[\bigl(x^3-z\mathfrak{c}_1(u+v)\bigr)\bigl(x^3-z\mathfrak{c}_1(u)\bigr)\bigl(x^3-z\mathfrak{c}_1(v)\bigr)=\Phi\bigl(x, f(x)\bigr),\]
where $z=(1, x, x^2)$. 

\vspace{1ex}

\noindent $(\mathrm{ii})$ For $u, v\in\mathbb{C}^3$, we have 
\[\mathfrak{c}_2(u+v)=\bigl(G(u+v)+\mathfrak{d}_2I_3\bigr)^{-1}\bigl(\mathfrak{d}_1\mathfrak{c}_3(u+v)+\mathfrak{d}_3\mathfrak{c}_1(u+v)+\mathfrak{d}\bigr).\]
\end{thm}

We consider the $3\times 3$ matrix $L$ defined in the proof of Lemma \ref{2.15.1}. 

\begin{cor}\label{2023.2.21.3}
Let $u={}^t(u_1,u_3,u_5)\in\mathbb{C}^3$. 
For $i=1,3,5$, we can express $\wp_{1,i}^V(Lu)$ and $\wp_{1,1,i}^V(Lu)$ in terms of $\wp(u_1)$, $\wp'(u_1)$, $\wp_{1,j}^C(u_3,u_5)$, and $\wp_{1,1,j}^C(u_3,u_5)$ with $j=1,3$ explicitly. 
\end{cor}

\begin{proof}
We have $Lu=L_1u_1+L_2{}^t(u_3,u_5)$. 
From Theorem \ref{2024.9.15.1}, for $i=1,3,5$, we can express $\wp_{1,i}^V(Lu)$ and $\wp_{1,1,i}^V(Lu)$ in terms of $\wp_{1,j}^V(L_1u_1)$, $\wp_{1,1,j}^V(L_1u_1)$, $\wp_{1,j}^V\bigl(L_2{}^t(u_3,u_5)\bigr)$, and $\wp_{1,1,j}^V\bigl(L_2{}^t(u_3,u_5)\bigr)$ with $j=1,3,5$. 
From Theorem \ref{2023.2.21.2}, for $j=1,3,5$, we can express $\wp_{1,j}^V(L_1u_1)$ and $\wp_{1,1,j}^V(L_1u_1)$ in terms of $\wp(u_1)$ and $\wp'(u_1)$.  
From Theorem \ref{2023.2.10.1}, for $j=1,3,5$, we can express $\wp_{1,j}^V\bigl(L_2{}^t(u_3,u_5)\bigr)$ and $\wp_{1,1,j}^V\bigl(L_2{}^t(u_3,u_5)\bigr)$ in terms of $\wp_{1,k}^C(u_3,u_5)$ and $\wp_{1,1,k}^C(u_3,u_5)$ with $k=1,3$. 
Therefore, we obtain the statement of the corollary. 
\end{proof}

\begin{rem}
We assume that one of the following three conditions is satisfied: 
\[a^2=b^2c^2,\qquad b^2=a^2c^2,\qquad c^2=a^2b^2,\]
where $a$, $b$, and $c$ are defined by (\ref{4.26}). 
Then the functions $\wp_{1,j}^C$ and $\wp_{1,1,j}^C$ with $j=1,3$ are expressed in terms of elliptic functions (cf. \cite[Theorem 5.9]{AB2022}). 
Therefore, in this case, we can express $\wp_{1,i}^V$ and $\wp_{1,1,i}^V$ with $i=1,3,5$ in terms of elliptic functions. 
\end{rem}

\vspace{2ex}

{\bf Acknowledgements.} This work was supported by JSPS KAKENHI Grant Number JP21K03296 and was partly supported by MEXT Promotion of Distinctive Joint Research Center Program JPMXP0723833165.

\end{document}